\documentclass[a4paper,11pt]{amsart}
\usepackage{amsmath, amsthm, amssymb, graphicx, amsrefs}
\usepackage{color}

\theoremstyle{plain}
\newtheorem{theorem}{Theorem}[section]
\newtheorem{lemma}[theorem]{Lemma}
\newtheorem{proposition}[theorem]{Proposition}
\newtheorem{corollary}[theorem]{Corollary}

\theoremstyle{definition}
\newtheorem{definition}[theorem]{Definition}
\newtheorem{example}[theorem]{Example}

\theoremstyle{remark}
\newtheorem{remark}[theorem]{Remark}

\numberwithin{equation}{section}

\newcommand{\R}{{\mathbb R}}
\newcommand{\N}{{\mathbb N}}
\newcommand{\Z}{{\mathbb Z}}
\newcommand{\Lp}{\mathcal{L}}
\newcommand{\lm}{\lambda}
\newcommand{\gm}{\gamma}
\newcommand{\Gm}{\Gamma}
\newcommand{\al}{\alpha}
\newcommand{\pt}{\partial_t}
\newcommand{\Deg}{{\mathrm{Deg}}}
\newcommand{\inte}{{\mathrm{int }}}
\newcommand{\cp}{{\mathrm{cap}}}

\newcommand{\as}[1]{\left\langle #1\right\rangle}

\newcommand{\ov}[1]{\overline{ #1}}
\newcommand{\ow}[1]{\widetilde{ #1}}
\newcommand{\oh}[1]{\widehat{ #1}}

\newcommand{\wt}{\widetilde}

\newcommand{\Hm}[1]{\leavevmode{\marginpar{\tiny%
$\hbox to 0mm{\hspace*{-0.5mm}$\leftarrow$\hss}%
\vcenter{\vrule depth 0.1mm height 0.1mm width \the\marginparwidth}%
\hbox to 0mm{\hss$\rightarrow$\hspace*{-0.5mm}}$\\\relax\raggedright
#1}}}

\providecommand{\eat}[1]{}

\begin{document}

\title[Coverings of graphs]{Coverings and the heat equation on graphs: stochastic incompleteness, the Feller property and uniform transience}

\author{Bobo Hua}
\address{B. Hua, School of Mathematical Sciences, Fudan University, 200433, Shanghai, China; Shanghai Center for Mathematical Sciences, Fudan University, Shanghai
200433, China}
\email{bobohua@fudan.edu.cn}

\author{Florentin M\"unch}
\address{F. M\"unch, Department of Mathematics, University of Potsdam, Potsdam, Germany.}
\email{chmuench@uni-potsdam.de}

\author{Rados{\l}aw K. Wojciechowski}
\address{R.K. Wojciechowski, Graduate Center of the City University of New York, 365 Fifth Avenue, New York, NY, 10016 and}
\address{York College of the City University of New York, 94-20 Guy R. Brewer Blvd., Jamaica, NY 11451.}
\email{rwojciechowski@gc.cuny.edu}
\date{\today}

\thanks{B.H. is supported by NSFC, grant no.\ 11401106. R.K.W. is supported by the Simons Foundation and PSC-CUNY Research Awards and thanks Fudan University and Hokkaido University for their generous hospitality while parts of this work were carried out.}

\begin{abstract} We study regular coverings of graphs and manifolds with a focus on properties of the heat equation. In particular, we look at stochastic incompleteness,
the Feller property and uniform transience; and investigate the connection between the validity of these properties on
the base space and its covering.  For both graphs and manifolds, we prove the equivalence of stochastic incompleteness 
of the base and that of its cover. Along the way we also give some new conditions for the Feller property to hold on graphs.
\end{abstract}

\maketitle
\tableofcontents

\section{Introduction}
Connections between coverings and properties of solutions to the heat equation have been studied previously by various authors for both
Riemannian manifolds \cite{Bro81, LS84, Bro85, Li86, Gri99, PS12} and graphs \cite{CY99, DM06}, among other works.
Here, we contribute to this investigation by looking at three recently developed properties of interest for the heat equation on infinite weighted 
graphs, namely, stochastic incompleteness, the Feller property and uniform transience; and study how these properties behave with respect to coverings.  

All three of these properties involve heat escaping to infinity in some sense.
Stochastic incompleteness (or non-conservativeness) concerns a loss in the total amount of heat at some time.  This has been studied rather thoroughly for
manifolds, see \cite{Gri99} for an overview, and, more recently, for graphs \cite{DM06, Woj08, Woj09, Web10, Hua11, Woj11, Hua12, KL12, GHM12, 
KLW13, Hua14, Fol14b, HL17, MW}.    The Feller property concerns heat vanishing at infinity.  It has been investigated for manifolds in \cite{Aze74, Yau78, Dod83, Hsu89, Dav92, PS12}, among
other works and, again more recently, for graphs in \cite{Woj17}.  Uniform transience is a strengthening of transience (as well as of the Feller property)  
which was recently introduced for graphs in \cite{KLSW17}, following previous work in \cite{BCG01, Win10, Kas10, Kas13}.  This is also related to the
notion of uniform subcriticality which has been studied for elliptic operators defined on domains in Euclidean space in \cite{Pin88} and, more recently, for Schr{\"o}dinger
operators on weighted graphs in \cite{KPP}.
Intuitively, if transience means that heat (or a random walker) escapes to infinity eventually, uniform transience means that it does so in all directions.  The only connection in general between these properties
is that uniform transience always implies the Feller property, see \cite{KLSW17}.  

\smallskip
In this note, we investigate how these three properties propagate between a base space and a regular covering of the space.  As we will see, in general, a base graph
is stochastically incomplete if and only if the cover is stochastically incomplete while for the Feller property and for uniform transience we show that 
if the base satisfies these properties, then so does the covering but not the other way around, see Theorem~\ref{t:covering}.  In the case of finitely many sheets, all three statements
become equivalences.

For both stochastic incompleteness and the Feller property, 
we are guided in this by previous work on Riemannian manifolds.  For stochastic incompleteness, this result is known for manifolds but the proof 
found in the literature uses stochastic partial differential equations see, for example, \cite{Elw82}. 
It was asked in \cite{PS12} to find a deterministic proof of this result. In this note, we provide such a proof in Theorem~\ref{thm:scequ}. We  use a
result of Li, see Theorem~\ref{thm:pli}, which reveals a simple connection between the heat kernel on the base and the heat kernel on the cover and can be derived from the arguments found in \cite{Bor00, Li12}. We first develop this approach for graphs, see Theorem~\ref{t:equality} and Theorem~\ref{t:covering}~(i).

Likewise, for the Feller property on manifolds, the recent paper of \cite{PS12} offers such a result with a slightly more difficult proof while it is rather a direct consequence of the connection between the heat kernels mentioned above, see Theorem~\ref{t:covering}~(ii).  
For uniform transience, we are not guided by any work on manifolds but rather 
exploit a connection between the various equivalent statements for uniform transience found in \cite{KLSW17} and the Green's function.

Before establishing these connections, we further explore the Feller property for graphs.  
We first show that the Feller property enjoys a certain uniformity with respect to time in Lemma~\ref{l:uniformity}.  
We then give some new conditions for the Feller property to hold.
Specifically, we first utilize the elliptic characterization of the Feller property to give a condition for the Feller property in terms of an inner degree
growth in Theorem~\ref{t:Feller1}.  
This greatly improves a result found in \cite{Woj17} where a parabolic viewpoint is used to show that a uniform bound on the vertex degree
implies the Feller property in analogy to \cite{Yau78, Dod83}.  Given that many criteria for the Feller property on Riemannian manifolds involve lower bounds
on the Ricci curvature, it is surprising that such lower bounds do not imply the Feller property in the discrete setting as we show in Proposition~\ref{p:curvature_feller}.
We then use the heat kernel estimates proven in \cite{BHY17} to give another
growth condition for the Feller property which connects the decay of the vertex measure and the growth of an intrinsic metric on the graph in Theorem~\ref{t:Feller2}.
The concept of an intrinsic metric was first introduced in full generality in \cite{FLW14} and has  found numerous applications in proving results analogous to those on Riemannian manifolds in the graph setting, see \cite{Fol11, BHK13, HKMW13, HKW13, Fol14a, Fol14b, Hua14, HK14, BKW15} and \cite{Kel15} for a 
survey of results in this direction.

\bigskip 

The structure of the paper is as follows. In Section~\ref{s:setting} we introduce our main setting of infinite weighted graphs and define the concepts of stochastic 
incompleteness, the Feller property and uniform transience in this context.  We also discuss the notion of a regular covering for weighted graphs.  In Section~\ref{s:Feller}
we have a closer look at the Feller property and prove the uniformity in time as well as the improved criteria for the Feller property to hold mentioned above.
In Section~\ref{s:coverings} we study the connections between a base and its covering with respect to these properties and also give some 
spectral consequences in this setting.
In Section~\ref{sec:manifold} we give a proof of the equivalence of stochastic incompleteness of a manifold and that of its cover.

\section{Setting and basic definitions} \label{s:setting}
\subsection{Laplacians and the heat equation}
We consider weighted graphs and graph Laplacians as in \cite{KL12} with no killing term and with the additional assumption that our graphs are locally finite.
That is, a \emph{graph} $G=(X,b,m)$ is a triple where $X$ is a countable set of \emph{vertices}, $b:X \times X \to [0,\infty)$ is an \emph{edge weight}
which satisfies $b(x,x)=0$, $b(x,y)=b(y,x)$ and $| \{y \ | \ b(x,y)>0 \} | < \infty$ and $m: X \to (0,\infty)$ is \emph{vertex measure} which can be
extended to all subsets of $X$ by countable additivity.  

For $x \in X$, we let the \emph{weighted degree} of $x$ be given by
\[ \Deg(x) = \frac{1}{m(x)}\sum_{y \in X} b(x,y). \]  
If $b(x,y)>0$, we say that $x$ and $y$ are \emph{connected} by an edge with weight $b(x,y)$
and write $x \sim y$.  For $x \in X$, we call the set $\{ y \ | \ y \sim x\}$ the \emph{neighborhood} of $x$.
We assume that all graphs are \emph{connected} in the usual sense of paths, that is, for all $x, y \in X$, there exists a sequence of vertices
$(x_i)_{i=0}^n$ such that $x=x_0$, $y=x_n$ and $x_i \sim x_{i+1}$ for all $i=0,1, \ldots n-1$.  We denote the usual combinatorial graph metric
by $d$, that is, $d(x,y):= \inf\{n\ | \ x=x_0\sim \ldots \sim x_n=y\}$.  Likewise, we will say that a subset of $X$ is \emph{connected} if
it is connected in the sense of paths which remain in the subset.
If $b(x,y) \in \{0,1\}$ and $m=1$, then we say that the graph
has \emph{standard edge weight and measure}.

We let $C(X) = \{ f: X \to \R \}$ denote the space of all real-valued functions on $X$ and let $\Lp:C(X) \to C(X)$ denote the \emph{formal Laplacian} 
which is given by 
\[ \Lp f(x) = \frac{1}{m(x)} \sum_{y \in X} b(x,y) (f(x) - f(y)). \]
If $C_c(X)$ denotes the finitely supported functions in $C(X)$ and 
\[ \ell^2(X,m) = \{ f \in C(X) \ | \ \sum_{x \in X} f^2(x) m(x)  < \infty \}\]
with inner product $\as{f,g} = \sum_{x \in X} f(x)g(x)m(x)$ and associated norm $\| f \| = \as{f,f}^{1/2}$
denotes the Hilbert space of square summable functions with respect to $m$, then we let $L$ denote the smallest self-adjoint extension of
$\Lp$ restricted to $C_c(X)$, see \cite{KL12, HKLW12, HKMW13} for more details.  

For $t>0$, we let $e^{-tL}$ denote the heat semigroup of $L$ and let $p_t(x,y)$ denote the \emph{heat kernel}
of the graph which is defined by
\[ e^{-tL}f(x) = \sum_{y \in X} p_t(x,y) f(y) m(y) \]
for all functions $f \in \ell^2(X,m)$.  We note that $u(x,t) = e^{-tL}f(x)$ is the minimal solution to the heat equation $(L + \pt)u =0$
with initial condition $u(x,0)=f(x)$ whenever $f\geq0$.  In particular, $p_t(x,y)$ is the smallest non-negative function which satisfies
$(L+ \pt)p_t(x,y)=0$, where the Laplacian is applied in either variable, and $p_0(x,y) = \oh{1}_x(y)$ where $\oh{1}_x = 1_x/m(x)$ is 
the delta function at $x$ divided by the measure at $x.$  Furthermore, as we assume that the graph is connected, $p_t(x,y)>0$
for all $t>0$, $x,y \in X$, see \cite{KL12}

By monotone approximation, the heat semigroup can be extended to all $\ell^p(X,m)$ for $p \in [1,\infty]$,
see \cite{KL12} for details.  In particular, the heat semigroup can be applied to the constant function 1, which is 1 on all vertices. This fact will be needed for the definition of stochastic incompleteness given below.

For vertices $x, y \in X$ we let $g(x,y)$ denote the \emph{Green's function} which is defined by
\[ g(x,y) = \int_0^\infty p_t(x,y) dt. \] 
Note that this function is either always infinite or always finite.  In the first case, a graph is called \emph{recurrent}, in the second, \emph{transient}.
An alternative definition for the Green's function is given via resolvents as follows:
\[ g(x,y) = \lim_{\alpha \to 0^+} (L+ \alpha)^{-1} \oh{1}_x(y). \]

For a sequence of vertices $(x_n)$, we write $x_n \to \infty$ as $n \to \infty$ if $(x_n)$ leaves every finite set eventually.  Furthermore, we let
\[ C_0(X) = \ov{C_c(X)}^{\| \cdot \|_\infty} \]
denote the set of functions \emph{vanishing at infinity} where $\| f \|_\infty = \sup_{x \in X} |f(x)|$.
Hence, $f \in C_0(X)$ if and only if $f(x_n) \to 0$ for every $x_n \to \infty$.

With these preparations we can define the three properties of the heat equation which we consider in this paper.
\begin{definition}
A graph $G=(X,b,m)$ is said to satisfy
\begin{itemize}
\item[(SI)]  \emph{Stochastic incompleteness} if for some (all) $x \in X$, some (all) $t >0$
\[ \sum_{y \in X} p_t(x,y)m(y) < 1.\]
\item[(FP)] The \emph{Feller property} if for some (all) $x \in X$, some (all) $t>0$
\[ p_t(x,y_n) \longrightarrow 0 \textup{ as } y_n \to \infty. \]
\item[(UT)] \emph{Uniform transience} if there exists a constant $C>0$ such that for all $x \in X$ 
\[ g(x,x) \leq C. \]
\end{itemize}
\end{definition}

\begin{remark}\label{r:introduction}
\begin{itemize} 
\item[(i)]  Note that all three properties have to do with heat escaping at infinity.  However, all of the properties have quite a different flavor.
In particular, both (SI) and (FP) depend strongly on the measure while (UT) does not.  In fact, if the inequality in the definition of (UT) holds
for one measure $m$, then it holds for all measures (with the same constant).  Furthermore, while (SI) and (UT) require a large growth on the graph,
(FP) can happen in the case of both large and small growth.  The only general implication that holds between these properties is that (UT) $\Longrightarrow$ (FP)
as noted in \cite{KLSW17} where (UT) is systematically introduced and studied.  However, note that the Green's function does not appear in \cite{KLSW17} 
but (UT) is rather introduced via several other equivalent conditions.  In particular, (UT) is equivalent to $\inf_x \cp(x) >0$ where $\cp(x)$ denotes the 
\emph{capacity} of $x$ which is defined via
\[ \cp(x) = \inf_{\varphi \in C_c(X), \varphi(x)=1} Q(\varphi) \]
where $Q(\varphi) = \frac{1}{2} \sum_{x,y \in X} b(x,y) (\varphi(x) - \varphi(y))^2$ denotes the \emph{energy} of $\varphi$.  Furthermore, (UT) is 
also equivalent to the existence of a constant $C>0$ such that $C \| \varphi \|_\infty \leq Q(\varphi)$ for all $\varphi \in C_c(X)$.  As pointed out by 
M. Schmidt, 
either of these
conditions is easily seen to be equivalent to our definition of (UT) by using general principles such as the resolvent formulation of the Green's function and 
the Green's formula.
\item[(ii)]  An equivalent formulation for (FP) is that $e^{-tL}:C_0(X) \to C_0(X)$, as such, this is also called the $C_0$-\emph{conservativeness} property.
\item[(iii)]  The fact that (UT) $\Longrightarrow$ (FP) mentioned above follows from another characterization of (UT) given in \cite{KLSW17}.  
Namely, (UT) is equivalent to the fact that the domain of the form associated to $L$ is contained in $C_0(X)$ for all measures $m$.  That is, 
\[ D(Q) = \ov{C_c(X)}^{\| \cdot \|_Q} \subseteq C_0(X) \]
for all measures $m$ where $\| \varphi \|_Q = (\|\varphi\|^2 + Q(\varphi))^{1/2}$.  Hence, if a graph satisfies (UT), then $e^{-tL}(C_c(X)) \subseteq D(Q)
\subseteq C_0(X)$ and (FP) follows by continuity of the semigroup with respect to the sup norm.
\item[(iv)]  It is always true that $\sum_{y \in X} p_t(x,y)m(y) \leq 1$.  In particular, if $\inf_x m(x) > 0$, then a graph automatically satisfies (FP) as
pointed out in \cite{Woj17}.  We will improve this result below to allow some decay to 0 on the part of $m$.
\item[(v)]  We mention that there are elliptic viewpoints for both (SI) and (FP) as follows: (SI) is equivalent to the existence of a positive, bounded function $v$ such that $\Lp v \leq \lm v$
for some $\lm <0$, see \cite{KL12}.   (FP) is equivalent to the existence of a positive function $v$ which vanishes at infinity such that $\Lp v \geq \lm v$ for some 
$\lm<0$, see \cite{Woj17}.  We will return to this later.
\item[(vi)]  It follows from the semigroup property and maximum principles that if the graph satisfies (SI) for some $x$ and some $t$, then it satisfies
(SI) for all $x$ and all $t$, see \cite{KL12}.  From the elliptic viewpoint for (FP), it is clear that (FP) satisfies the same property with respect to $x$. 
We will establish that (FP) satisfies an even stronger uniformity with respect to $t$ below, see Lemma~\ref{l:uniformity}.
\end{itemize}
\end{remark}

\subsection{Regular coverings} 
We now make precise the notion of a regular covering in the setting of weighted graphs.  We consider a graph as a 1-dimensional simplicial complex which is a
metric space with respect to the combinatorial graph metric.  
\begin{definition}
We say that a graph
$\ow{G}= (\ow{X},\ow{b},\ow{m})$ is a \emph{regular covering} of $G=(X,b,m)$ if
$(\ow{X},\ow{b},\ow{m})$ is a regular covering space in the topological sense and if the edge weights and measures are such that the deck transformations are
graph isomorphisms.   That is, there exists an onto map 
\[ \pi: \ow{X} \to X\] 
which, for every $\ow{x} \in \ow{X}$, is a graph isomorphism on the neighborhood of $\ow{x}$  and which satisfies 
\[ \ow{b}(\ow{x},\ow{y})= b(x,y) \qquad \textup{ and } \qquad \ow{m}(\ow{x}) = m(x)\] 
for all $\ow{x},\ow{y} \in \ow{X}$ with $\ow{x} \sim \ow{y}$, $\pi(\ow{x}) = x$ and $\pi(\ow{y}) = y$.  We call $\ow{G}$ the \emph{cover} 
and $G$ the \emph{base} in this case.   The set $\pi^{-1}(x)$ is called the \emph{fiber} over $x\in X$ and the cardinality of this set is referred to 
as the \emph{number of sheets} of the covering.
\end{definition}

In particular, note that $\ow{\Lp}(f \circ \pi)(\ow{x}) = \Lp f(x)$ for all $x \in X, \ow{x} \in \ow{X}$ such that $\pi(\ow{x})=x$ and all $f \in C(X)$
where $\ow{\Lp}$ $(\Lp$, respectively) denotes the Laplacian on $\ow{G}$ ($G$, respectively).
Furthermore, as the covering is regular, for every $x \in X$ and all $\ow{x}_1, \ow{x}_2 \in \pi^{-1}(x)$, there exists a deck transformation 
$\gamma$ which is a graph isomorphism such that $\gamma(x_1) = x_2$.  We will denote the set of all deck transformations by $\Gamma$.
In particular, 
\[ \ow{p}_t(\ow{x},\ow{y}) = \ow{p}_t(\gamma (\ow{x}), \gamma (\ow{y})) \]
for all $\gamma \in \Gamma, t \geq 0$ and $\ow{x}, \ow{y} \in \ow{X}$.

\section{The Feller property}\label{s:Feller}
In this section we have a closer look at the Feller property (FP).  First, we show that (FP) satisfies a uniformity in both space and time.  We then give
some new criteria for (FP) to hold on graphs.

\subsection{Uniformity}  We show that if the heat kernel vanishes at infinity for one $t$, then it does so for all $t$.  In fact, we
show that (FP) is equivalent to an even stronger statement with respect to time.  In order to do so, we utilize the semigroup property which
states that $e^{-(s+t)L} = e^{-sL}e^{-tL}$ or, in terms of the heat kernel,
\[ p_{s+t}(x,y) = \sum_{z \in X} p_s(x,z)p_t(z,y)m(z). \]

Recall that a graph satisfies (FP) if $p_t(x, y_n) \to 0$ for all $t>0$ and all $x\in X$ as $y_n \to \infty$.  We will say that a graph satisfies the \emph{uniform Feller property} or (UFP) if
\[ \max_{t \in [0,T]} p_t(x,y_n) \longrightarrow 0 \textup{ as } y_n \to \infty \]
for all $x \in X$, all $T>0$.
We now show that for both of these properties, it suffices that the heat kernel vanishes at infinity at only one time and one vertex. 

\begin{lemma}\label{l:uniformity}
Let $G=(X,b,m)$ be a graph with heat kernel $p$.  The following statements are equivalent:
\begin{itemize}
\item[(i)]  $p_{t_0}(x_0,y_n) \longrightarrow 0$ as $y_n \to \infty$ for some $t_0>0$, some $x_0 \in X$. 
\item[(ii)] $G$ satisfies (FP).
\item[(iii)] $G$ satisfies (UFP).
\end{itemize}
\end{lemma}
\begin{proof}
(i) $\Longrightarrow$ (ii):  Let $x \in X$ and $t < t_0$.  It then follows from the semigroup property that
\begin{align*}
p_{t_0}(x_0,y_n) &= \sum_{z \in X} p_{t_0-t}(x_0,z) p_t(z,y_n) m(z)\\
&\geq p_{t_0-t}(x_0,x)p_t(x,y_n) m(x).
\end{align*}
Therefore,
\[ p_t(x,y_n) \leq \frac{p_{t_0}(x_0,y_n)}{p_{t_0-t}(x_0,x)m(x)} \longrightarrow 0 \textup{ as } y_n \to \infty \]
for all $x \in X$ and $t < t_0$.
Taking finite sums yields $e^{-tL}(C_c(X)) \subseteq C_0(X)$ for $t < t_0$. A density argument yields $e^{-tL}: C_0(X) \to C_0(X)$ for $t<t_0$ and, finally, the semigroup property gives $e^{-tL}: C_0(X) \to C_0(X)$ for all $t\geq 0$.

(ii) $\Longrightarrow$ (iii):  For a fixed $x \in X$, note that $u(y,t) = e^{-t\Deg(x)}1_x(y)$ is a subsolution for the heat equation, that is,
$(\Lp + \pt) u(y,t) \leq 0$, with $u(y,0) = 1_x(y)$.  By applying a maximum principle such as Proposition~2.2 in \cite{Woj17}, it follows that
$e^{-t\Deg(x)}1_x \leq e^{-tL}1_x$, that is,
\[ e^{-t\Deg(x)}1_x(y) \leq p_t(x,y)m(y). \]
Now, for $T>0$ and $t \in [0,T]$, we get that
\begin{align*}
p_T(x,y_n) &= \sum_{z \in X}p_{T-t}(x,z)p_t(z,y_n)m(z) \\
&\geq e^{-(T-t)\Deg(x)}p_t(x,y_n).
\end{align*}
Therefore,
\[ p_t(x,y_n) \leq e^{(T-t)\Deg(x)}p_T(x,y_n) \leq e^{T\Deg(x)}p_T(x,y_n) \]
for all $t\in [0,T]$.  The conclusion then follows.

(iii) $\Longrightarrow$ (i):  This is clear.
\end{proof}

\subsection{Degree criteria}  We now prove criteria for the Feller property (FP) involving vertex degree quantities. 
Any graph for which the weighted degree $\Deg(x)= \frac{1}{m(x)} \sum_{y \in X} b(x,y)$ is a bounded function on $X$ 
satisfies (FP), see Theorem~4.2 in \cite{Woj17}.  We note that this condition is equivalent to $L$ being a bounded operator on $\ell^2(X,m)$, see \cite{KL12}.

This result was obtained by using the parabolic perspective and gives a counterpart to the result on manifolds which states that if the Ricci curvature
is uniformly bounded from below, then the manifold satisfies (FP), see \cite{Yau 78, Dod83}.  However, in the manifold case, the optimal result for Ricci 
curvature is obtained by using probabilistic methods and allows for some rate of decay, see \cite{Hsu89} and further discussion in Subsection~\ref{s:curvature}.

\bigskip
In order to prove our criteria, we take advantage of the elliptic perspective on the Feller property first pointed out for manifolds in \cite{Aze74} which formally
carries over to the graph setting.  That is, by combining Theorems~3.3~and~3.6 in \cite{Woj17}, we
obtain that a graph satisfies (FP) if and only if there exists a positive function $v \in C_0(X)$ such that
\[ \Lp v \geq \lm v \]
for $\lm <0$.

For a vertex $x_0 \in X$, we let $S_r := S_r(x_0) := \{ x \ | \ d(x,x_0)=r \}$ where $d$ denotes the standard combinatorial graph metric.  We let 
\[ D(r) = \max_{x \in S_r} \Deg(x) \]
denote the maximal degree on a sphere. 
For $x \in S_r$, we let $\Deg_{\pm}(x) = \frac{1}{m(x)}\sum_{y \in S_{r\pm1}} b(x,y)$ denote the \emph{outer} and \emph{inner degree} of $x$
and let
\[ D_{\pm}(r) = \max_{x \in S_r} \Deg_\pm(x) \quad \textup{ and } \quad d_{\pm}(r) = \min_{x \in S_r} \Deg_\pm(x). \]
\begin{theorem}\label{t:Feller1}
Let $G=(X,b,m)$ be a graph.
\begin{itemize}
\item[(i)] If for some vertex $x_0 \in X$
\[ \sum_{r} \frac{1}{D_-(r)} = \infty, \]
then the graph satisfies (FP).
\item[(ii)] If for some vertex $x_0 \in X$
\[ \sum_r \frac{D(r)-d_{-}(r)+1}{d_{-}(r)} < \infty, \]
then the graph does not satisfy (FP).
\end{itemize}
\end{theorem}

This yields the following immediate corollary. 
\begin{corollary}Let $G=(X,b,m)$ be a graph. If for some vertex $x_0 \in X$
\[ D_-(r)=O(r), \quad \mathrm{as}\ r\to \infty,\]
then the graph satisfies (FP).
\end{corollary}
\begin{remark}
We contrast the conditions in Theorem~\ref{t:Feller1} with some related criteria for (SI).  Namely, in \cite{Woj08, Woj09}, it is shown that for graphs with standard edge weights and measure if $\sum \frac{D_{-}(r)+1}{d_+(r)} < \infty$, then a
graph satisfies (SI).  This was later improved in \cite{Hua11} to $\sum_r \max_{x \in S_r} \frac{\Deg_-(x)}{\Deg_+(x)} <\infty$ by using the weak Omori-Yau maximum principle for (SI).  Furthermore, in \cite{Woj11} it is shows that if $\sum_r \frac{1}{D_+(r)} = \infty$, then a graph does not satisfy (SI).  

Note that the conditions
for (FP) and (SI) are opposite in some sense.  The reason is that (SI) requires large growth while (FP) holds either due to large growth or small
growth.  The conditions presented here for (FP) have do with small growth.
\end{remark}
\begin{proof}[Proof of Theorem~\ref{t:Feller1}]
As mentioned above, (FP) is equivalent to the existence of a positive function $v \in C_0(X)$ such that $\Lp v \geq \lm v$ for $\lm<0$.  

For the proof of (i), we construct such a function depending only on the distance to $x_0$.  That is, let $v(x_0) = 1$ and for any
$x \in S_r$ for $r\geq 1$, let 
\[ v(r):= v(x) := \prod_{i=1}^r \frac{D_-(i)}{D_-(i) - \lm}. \]
As $\lm<0$, $v$ is decreasing as the distance to $x_0$ increases and as 
\[ \frac{1}{v(x)} = \prod_{i=1}^r \left( 1 - \frac{\lm}{D_-(i)} \right) \longrightarrow \infty \textup{ as } r \to \infty \]
by the assumption that $\sum_r \frac{1}{D_-(r)}=\infty$, it follows that $v \in C_0(X)$.  

We observe that $\Lp v(x_0) \geq 0 \geq \lambda v(x_0)$.
Finally, for $x \in S_r$, $r\geq1$, 
we get that
\begin{align*}
\Lp v(x) &= \Deg_+(x)(v(r)-v(r+1)) + \Deg_-(x)(v(r)-v(r-1)) \\
&\geq \Deg_-(x)v(r)\left(1-\frac{D_-(r) - \lm}{D_-(r)}\right) \\
&= \Deg_-(x)v(r)\left( \frac{\lm}{D_-(r)} \right) \\
&\geq \lm v(r) = \lm v(x).
\end{align*}
This finishes the proof of (i).

For (ii), let $v>0$ satisfy $\Lp v \geq \lm v$ for $\lm<0$.  Such a positive function exists as the resolvent is positivity improving in the case that the graph is connected, see \cite{KL12}.  Let $w(0) = v(x_0)$ and, for $r\geq1$,
\[ w(r) = \prod_{i=1}^r \left( \frac{d_-(i)}{D(i)-\lm} \right) v(x_0). \]
Note that as
\[ \frac{1}{w(r)} = \prod_{i=1}^r \left( 1 + \frac{D(i) - d_-(i) - \lm}{d_-(i)} \right) \not \longrightarrow \infty \textup{ as } r \to \infty \]
since $\sum_r \frac{D(r)-d_{-}(r)+1}{d_{-}(r)} < \infty$ it follows that $w \not \in C_0(X)$.  

We now claim by induction on $r$ that $v \geq w$.  For $r=0$ this is clear by definition.  Now, assume that $v(y) \geq w(r-1)$
for all $y \in S_{r-1}$ and let $x \in S_r$.  Then
\begin{align*}
\lm v(x) \leq \Lp v(x) &\leq \Deg(x) v(x) - \frac{1}{m(x)} \sum_{y \in S_{r-1}} b(x,y) v(y) \\
&\leq \Deg(x) v(x) - \Deg_-(x) w(r-1)
\end{align*}
from which it follows that
\[ v(x) \geq \left(\frac{\Deg_-(x)}{\Deg(x)-\lm} \right)w(r-1) \geq \left( \frac{d_-(r)}{D(r)-\lm} \right) w(r-1) = w(r). \]
This completes the proof as it follows that $v \not \in C_0(X)$.
\end{proof}

\subsection{Curvature and the Feller property}\label{s:curvature}
As noted above, Theorem~\ref{t:Feller1} (i) extends Theorem 4.2 from  \cite{Woj17} which gives (FP) in the case of uniformly bounded degree. 
The result on bounded degree in \cite{Woj17}, at least in terms of the proof, is an analogue to \cite{Yau78, Dod83} stating that any Riemannian manifold with Ricci curvature uniformly bounded from below will satisfy (FP).  

An optimal criterion for (FP) in terms of Ricci curvature in the setting of Riemannian manifolds is proven via probabilistic techniques in
\cite{Hsu89}.  Namely, Hsu shows that if $\kappa(r)$ is a lower bound on the Ricci curvature of a geodesic ball of radius $r$,
then
\[ \int^\infty \frac{1}{\sqrt{|\kappa(r)|}} dr = \infty \]
implies (FP). 

Recently, there has been a tremendous interest in various notions of curvature, and especially of Ricci curvature, for graphs, see, for example, \cite{Sch99, LY10,LLY11,BJL12, JL14, BHLLMY15, HL16, Mun17, MW, johnson2015discrete, Oll09,
ni2015ricci,maas2017entropic,rubleva2016ricci,
eldan2017transport,kempton2017relationships,
yamada2017curvature,gao2016one,lin2012ricci,
liu2017rigidity,gao2016curvature,liu2017distance,
lin2013ricci,lin2015equivalent,munch2014li,
liu2016bakry,cushing2016bakry,saucan2009combinatorial,
bhattacharya2015exact,fathi2016entropic,
ollivier2012curved,paeng2012volume,
sandhu2015graph,wang2014wireless,chung2014harnack,
fathi2018curvature,erbar2012ricci}.
In particular, two notions have been most prominently explored for finding analogues to results in the setting of Riemannian manifolds for graphs:
the Bakry-{\'E}mery approach having its origins in \cite{BE85} and the coarse Ollivier-Ricci curvature originating in the work \cite{Oll09}.
As such, given the results on Riemannian manifolds mentioned above, it would seem natural to ask if there is a Ricci curvature criterion for (FP) using 
these new curvature notions. 
In this subsection we discuss that this is, in general, not the case.  In particular, we show that there exist graphs
which satisfy arbitrary lower curvature bounds for both the Ollivier-Ricci and Bakry-{\'E}mery curvatures but do not satisfy (FP).

\bigskip

We first briefly discuss the definitions of the two curvatures mentioned above.
First, the Ollivier-Ricci curvature originally defined via optimal transport in \cite{Oll09} and modified in \cite{LLY11} was recently extended to general graph Laplacians
in \cite{MW}.  Although, we do not give the definition here, we mention that this curvature can be calculated explicitly for large classes of graphs.
In particular, for any graph satisfying $X = \N_0$ with $b(x,y)>0$ if and only if $|x-y|=1$, it follows that $\kappa(r):=\kappa(r-1,r)$, the curvature between adjacent 
vertices $r-1$ and $r$, can be calculated as
\begin{align}\label{e:curvature}
\kappa(r) &= \frac{b(r-1,r)-b(r-1,r-2)}{m(r-1)} - \frac{b(r,r+1)-b(r,r-1)}{m(r)}\\
&=\big(d_+(r-1) - d_-(r-1) \big) - \big( d_+(r)-d_-(r) \big)\nonumber
\end{align}
where we let $d_+(r)=b(r,r+1)/m(r)$ and $d_-(r)=b(r-1,r)/m(r)$ as above and let $d_-(0)=0$.

Secondly, Bakry-{\'E}mery curvature is defined via a Bochner formula as follows. The maximal lower Bakry-{\'E}mery curvature bound $K_{BE} \in C(X)$ is given by
\[ K_{BE}:= \sup\{K \in C(X): \Gamma_2(f,f) \geq K \Gamma_1(f,f), \forall f \in C(X)\} \]
where $\Gamma_0(f,g):=f\cdot g$ and for $k\geq 1$,
\[ \Gamma_k(f,g) := - \Lp \Gamma_{k-1}(f,g) + \Gamma_{k-1}(\Lp f, g) + \Gamma_{k-1}( f, \Lp g). \]

Using these formulas, we can construct examples of graphs which do not satisfy (FP) with both Ollivier-Ricci and Bakry-{\'E}mery  curvature satisfying arbitrary lower bounds.

\begin{proposition}\label{p:curvature_feller}
For every sequence $(k_r)_{r \in \N}, k_r \in \R$, there exists a graph with  $X= \N_0$ which does not satisfy (FP) such that both $\kappa(r) \geq k_r$ and $K_{BE}(r) \geq k_r$.
\end{proposition}
\begin{proof}
We let $X = \N_0$ with $b(x,y)>0$ if and only if $|x-y|=1$.  As above, we let $d_+(r)= b(r,r+1)/m(r)$ and $d_-(r)=b(r-1,r)/m(r)$ and 
we first choose $b$ and $m$ to satisfy $d_+(r)=1$ and $\sum_r \frac 1 {d_-(r)}<\infty.$
 Then, applying Theorem~\ref{t:Feller1} shows that the graph does not satisfy (FP).

Furthermore, for every $k_r \in \R$, we can make the choices 
$d_-(r)-d_-(r-1) \geq k_r$. Then, 
$\kappa(r) \geq k_r$ follows from \eqref{e:curvature}.

For Bakry-{\'E}mery curvature, we further refine the choices of $d_-(r)$ to be increasing and to satisfy
$d_-(r)-d_-(r-1) \geq 2(k_r \vee k_{r-1})$.
A straight-forward computation gives that $K_{BE}(r) \geq k_r$ is equivalent to
$W_-(r) \geq 0$ and $W_+(r) \geq 0$ and $W_-(r)W_+(r) \geq 4d_-(r)d_+(r)$ with
\[
W_-(r) := -d_-(r-1) + 3d_+(r-1) + d_-(r) - d_+(r) - 2k_r
\]
and
\[
W_-(r) := -d_+(r+1) + 3d_-(r+1) + d_+(r) - d_-(r) - 2k_r,
\] 
see Section~2 in \cite{HM17}.
Since $d_-(r) - d_-(r-1) \geq 2 k_r$, we have $W_-(r) \geq 2$. Since $d_-(r+1) - d_r \geq 2 k_r$, we have
$W_+(r) \geq 2d_-(r+1)$. Since $d_-(r)$ is increasing, we obtain $W_-(r)W_+(r) \geq 4d_-(r)d_+(r)$ which proves that $K_{BE}(r) \geq k_r$.
\end{proof}

\begin{remark}
We note that in the example above the curvatures turn out to be positive.
In fact, with a bit more effort, we can show that for any sequence $k_r \in \R$, there exist graphs such as above which do not satisfy (FP) 
and which have Ollivier-Ricci curvature $\kappa(r)=k_r$.  This is surprising given the manifold case were, for example, all Cartan-Hadamard manifolds satisfy (FP), 
see \cite{Aze74, PS12}.

To show this, we again take $X = \N_0$ with $b(x,y)>0$ if and only if $|x-y|=1$.
We first set $m(0)=1$ and $b(0,1)=2$ giving $\Deg(0)=2$.
Observe that by iterating \eqref{e:curvature}, $\kappa(r)=k_r$ for all $r \in \N$ is equivalent to
\[
\frac{b(r,r+1) - b(r,r-1)}{m(r)} = \Deg(0)- \sum_{j=1}^{r} k_j =: C_r.
\]
The idea of the construction relies on the following two observations:
First, if we choose $m(r)$ small enough, then we can guarantee $b(r,r+1)$ is uniformly bounded above and below by a constant. 
We can, for example, set $m(r)$ such that $|C_r| m(r) \leq 2^{-r}$ and $b(r,r+1):=b(r,r-1) + C_r m(r)$ yielding $|b(r,r+1)-b(r-1,r)| \leq 2^{-r}$ which gives $b(r,r+1) \in \left[1, 3\right]$ for all $r \in \N_0$. Moreover, the inductive definition of $b(r,r+1)$ guarantees that $\kappa(r)=k_r$.

Second, if we choose $m(r)$ small enough, then, we can guarantee 
\[
\sum_{r=1}^\infty m(\{r,r+1,\ldots\}) < \infty.
\]
This, in particular, holds true if $m(r)<2^{-r}$.

To satisfy both of these conditions, we can simply set $m(r):= \frac{2^{-r}}{1+|C_r|}$.

We easily see that $\sum_r \frac 1 {b(r,r+1)} = \infty$ since $b(r,r+1) \leq  3 $. Moreover, since $b(r,r+1)\geq 1$,
\[
\sum_r \frac {m(\{r,r+1,\ldots\})}{b(r,r-1)} \leq  \sum_r m(\{r,r+1,\ldots\}) < \infty.
\]
This implies that the graph does not satisfy the Feller property due to Theorem~4.13 in \cite{Woj17}.

\end{remark}

\subsection{An intrinsic metric criterion}  As previously mentioned, if the vertex measure is uniformly bounded from below by a positive constant, then the graph satisfies (FP).  We now improve this by allowing the measure  to decay to 0.  The rate of decay will involve the use of intrinsic metrics and a heat kernel estimate obtained in \cite{BHY17}.

\bigskip

We first introduce some concepts related to intrinsic metrics.  For the full theory of intrinsic metrics for non-local Dirichlet forms, which extends the framework of
local Dirichlet forms with killing term in \cite{Stu94}, see \cite{FLW14}.  For other applications for graphs, see the survey \cite{Kel15}.

\begin{definition}
We say that a metric $\rho:X \times X \to [0,\infty)$ is \emph{intrinsic} if
\[ \sum_{y \in X} b(x,y) \rho^2(x,y) \leq m(x) \]
for all $x \in X$.
We say that an intrinsic metric has \emph{finite jump size} $j>0$ if $\rho(x,y) \leq j$ for all $x \sim y$.   Finally, an intrinsic metric is called
\emph{proper} if all balls defined with respect to $\rho$ are finite.
\end{definition}
\begin{example}\label{e:intrinsic}
A standard example for graphs first found in \cite{Hua11a} is to let 
\[ \rho(x,y) = \left( \max \{ \Deg(x), \Deg(y) \} \right) ^{-1/2} \]
for all $x \sim y$ and then extend this to all vertices via paths.  It will have finite jump size if the weighted degree is uniformly bounded
from below and will be proper given that the weighted degree does not grow too rapidly.
\end{example}

We now recall a heat kernel estimate which follows from the Davies-Gaffney-Girgor'yan Lemma for graphs proven in  \cite{BHY17} 
(see also \cite{Pan93, Dav93, Dav93b, Del99, MS00, Sch02, Fol11, BHY15} for earlier work).  Namely, if $\rho$
is a proper intrinsic metric with finite jump size $j$, we get the following estimate for the heat kernel:
\begin{equation}\label{e:estimate}
 p_t(x,y) \leq \frac{1}{\sqrt{m(x)m(y)}}  \exp{(- \zeta_j(t,\rho(x,y)))}
\end{equation}
where
\[ \zeta_j(t,r) = \frac{1}{j^2}\left( jr \cdot \textup{arsinh} \left(\frac{jr}{t}\right) - \sqrt{t^2 + (jr)^2} + t \right). \]

We note that $\rho$ is not needed to be proper for the result above to hold.  However, we need that $y_n \to \infty$ if and only if $\rho(x, y_n) \to \infty$
for our result below and at this point we need the metric to be proper.

\begin{theorem}\label{t:Feller2}
Let $G=(X,b,m)$ be a graph with a proper intrinsic metric $\rho$ with finite jump size $j>0$.  If for some $x_0 \in X$ and some $C>0$ one has
\begin{equation*}
{- \log m(y)} \leq \frac{2 \rho(x_0,y)}{j} (\log \rho(x_0,y) + C)
\end{equation*}
for all $y \in X$, then $G$ satisfies (FP). 
\end{theorem}

\begin{proof}[Proof of Theorem~\ref{t:Feller2}]
We aim to show that $p_t(x_0,y) \to 0$ as $y \to \infty$ by using \eqref{e:estimate}. We remark that due to Lemma~\ref{l:uniformity}, it suffices to show that $p_t(x_0,y) \to 0$ for some small $t>0$.
First, we note that
\[
\zeta_j(t,r) \geq \frac r j \log \left(\frac {jr}{t}\right) - \frac r j \geq \frac r j (\log r + C+1)
\]
if $t>0$ is chosen small enough, where the first inequality follows from
$\textup{arsinh}(\alpha) \geq \log \alpha$ and $\alpha - \sqrt{\alpha^2 + \beta^2} \geq -\beta$ for $\alpha,\beta >0$.

Hence by \eqref{e:estimate},
\begin{align*}
\log p_t(x_0,y) &\leq -\frac 1 2 \log m(x_0) - \frac 1 2 \log m(y) - \zeta_j(t,\rho(x_0,y)) \\
&\leq -\frac 1 2 \log m(x_0) - \frac 1 2 \log m(y) -  \frac {\rho(x_0,y)}j  (\log \rho(x_0,y) + C+1)\\
& \leq -\frac 1 2 \log m(x_0) - \frac {\rho(x_0,y)}j
\end{align*}
where the last estimate follows by assumption.
This implies that $p_t(x_0, \cdot) \in C_0(X)$ since $\rho$ is proper which finishes the proof.
\end{proof}

\begin{example}
We give an example for which Theorem~\ref{t:Feller1} does not apply but Theorem~\ref{t:Feller2} does.  
For two positive functions $f, g: X \to (0,\infty)$ we will write $f \sim g$ if there exist positive constants $c_1, c_2>0$
such that $c_1 f(x) \leq g(x) \leq c_2 f(x)$ for all $x \in X$.

Let $X=\N$ with $b(x,y) = 1$ if $|x-y|=1$ and 0 otherwise and $m(r) \sim 1/r^2$ for $r \in \N$.  Then $\Deg_-(r) \sim r^2$ so that Theorem~\ref{t:Feller1} does not apply.
However, by using the intrinsic metric in Example~\ref{e:intrinsic}, we get that $\rho(r,r+1) \sim 1/r$ so that $\rho(0,r)\sim \log r$.  Hence, the metric is 
proper and has finite jump size $j$.  Therefore,
\[
-\log m(r) \sim \log r \leq C \rho(0,r)
\]
for all $r \in \N$ so that Theorem~\ref{t:Feller2} shows that such a graph satisfies (FP).
\end{example}

\section{Coverings and the heat equation}\label{s:coverings}
We now prove some connections between properties of the heat kernel on a graph and a covering of the graph.  We will see that a graph is stochastically incomplete
if and only if a covering of the graph is stochastically incomplete and the same is true for the Feller property and uniform transience in the case of finitely many sheets.  In contrast, for coverings with infinitely many sheets, 
only one implication holds for both the Feller property and for uniform transience, namely,
if the base satisfies (FP) or (UT), then the cover will also satisfy (FP) or (UT).  We show by example that the other implications do not hold.

\subsection{Heat kernel on the base and that on the cover}
In order to prove these results, we show a very simple relation between the heat kernel on the base and the heat kernel on the cover following work on manifolds
found in \cite{Bor00, Li12}.
\bigskip

Given a regular covering $\pi: \ow{X} \to X$ and the heat kernel $\ow{p}$ on $\ow{G}$, we define a new function on the base graph as follows:
\[ q_t(x,y) = \sum_{\ow{y} \in \pi^{-1}(y)} \ow{p}_t(\ow{x},\ow{y}) \]
for $t\geq0$, $x,y \in X$ where $\ow{x} \in \pi^{-1}(x)$.
We will prove that $q_t(x,y) = p_t(x,y)$ where $p$ is the heat kernel on $G$.  
Throughout, we will let $L$ denote the Laplacian on $G$ while $\ow{L}$ will denote the Laplacian on $\ow{G}$.
We will also put a subscript to indicate in which variable the Laplacian is being applied when necessary.

We first show that $q_t(x,y)$ is well-defined.
That is, if $\ow{x}_1,\ow{x}_2 \in \pi^{-1}(x)$, then we must show that 
$\sum_{\ow{y} \in \pi^{-1}(y)} \ow{p}_t(\ow{x}_1,\ow{y})= \sum_{\ow{y} \in \pi^{-1}(y)} \ow{p}_t(\ow{x}_2,\ow{y})$.
As the covering is regular, there exists a deck transformation $\gamma$ such that $\gamma(\ow{x}_1) = \ow{x}_2$.
It then follows that
\begin{align*} 
\sum_{\ow{y} \in \pi^{-1}(y)} \ow{p}_t(\ow{x}_2,\ow{y}) &= \sum_{\ow{y} \in \pi^{-1}(y)} \ow{p}_t(\gamma (\ow{x}_1),\ow{y}) \\
&= \sum_{\ow{y} \in \pi^{-1}(y)} \ow{p}_t( \ow{x}_1,\gamma^{-1} (\ow{y})) = \sum_{\ow{y} \in \pi^{-1}(y)} \ow{p}_t( \ow{x}_1, \ow{y})
\end{align*}
since $\gamma$ is an isomorphism and $\Gamma$ acts transitively on each fiber.  Thus, $q$ is well-defined.  Furthermore, note that $q$ is finite as
\[ q_t(x,y) = \sum_{\ow{y} \in \pi^{-1}(y)} \ow{p}_t(\ow{x},\ow{y}) = \frac{1}{m(y)} \sum_{\ow{y} \in \pi^{-1}(y)} \ow{p}_t(\ow{x},\ow{y}) \ow{m}(\ow{y}) 
\leq \frac{1}{m(y)}. \]

Next we will show that $q$ satisfies the heat equation on $G=(X,b,m)$.  In order to do so, we need to justify the interchange of the derivative and summation
via the use of the dominated convergence theorem.  Therefore, we need to bound the absolute value of the derivatives of $\ow{p}$ by a summable
function independently of $t$.  This is a general phenomenon and does not have to do with coverings so we present it as such and then apply it to $q$.

\begin{lemma}\label{l:uniform}
For every $T>0$ and $x \in X$, there exists $f_x:=f_{x,T} \in \ell^1(X,m)$ such that
\[  \max_{t \in [0,T]} p_t(x,y) \leq f_{x}(y) \]
for all $y \in X$.
\end{lemma}
\begin{proof}
Let $\oh{1}_x(y) = 1_x(y)/m(x)$ and note that $p_0(x,y) = \oh{1}_x(y)$.  As $p_t(x,y)$ satisfies
the heat equation in either variable,
\begin{align*}
p_t(x,y) &= \oh{1}_x(y) + \int_0^t \partial_s p_s(x,y) ds \\
&\leq \oh{1}_x(y) + \int_0^T |  \partial_s p_s(x,y)| ds \\
&=  \oh{1}_x(y) + \int_0^T | L_x p_s(x,y)| ds \\
&\leq  \oh{1}_x(y) + \int_0^T \left(\Deg(x) p_s(x,y) + \frac{1}{m(x)} \sum_{z \in X} b(x,z)p_s(z,y) \right) ds \\
&=: f_x(y). 
\end{align*}
Since this holds for all $t \in [0,T]$, it is clear that $\max_{t \in [0,T]} p_t(x,y) \leq f_x(y)$.

Now, by applying Fubini's Theorem and using that $\sum_{y \in X} p_s(x,y) m(y) \leq 1$, we get that
\begin{align*}
\sum_{y \in X} f_x(y)m(y) \leq 1 + 2T \Deg(x)
\end{align*}
so that $f_x \in \ell^1(X,m)$.
\end{proof}

We now use the lemma above to show that $q$ satisfies the heat equation on $G$.
\begin{lemma}\label{l:heat_equation}
 For all $x,y \in X$, $t\geq0$, we have that
\[ (L + \pt) q_t(x,y) = 0 \] 
where the Laplacian is applied in either variable.   Furthermore, 
\[ q_0(x,y) = p_0(x,y). \]
\end{lemma}
\begin{proof}
Recall that by the definition of the covering, we have that $(L + \pt) (\ow{p} \circ \pi^{-1})=0$ where $L$ is applied in either variable.  Therefore,
in order to show that $q$ satisfies the heat equation on $X$, it suffices to show that
\[ \pt \sum_{\ow{y} \in \pi^{-1}(y)} \ow{p}_t(\ow{x},\ow{y}) = \sum_{\ow{y} \in \pi^{-1}(y)} \pt \ow{p}_t(\ow{x},\ow{y}), \]
that is, that the summation and derivative commute.  For this, it suffices to observe that $\ow{p}$ is differentiable in $t$,
that, as noted above, $\sum_{\ow{y} \in \pi^{-1}(y)} \ow{p}_t(\ow{x},\ow{y}) \leq 1/m(y)$ for all $t$
and that for any $T>0$ and $t \in [0,T]$ with $f_{\ow{x}}$ as in Lemma~\ref{l:uniform}, we have
\begin{align*} 
| \pt \ow{p}_t(\ow{x}, \ow{y}) | &= | \ow{L}_{\ow{x}} \ow{p}_t(\ow{x},\ow{y}) |  \\
&\leq  \Deg(x) f_{\ow{x}}(\ow{y}) + \frac{1}{\ow{m}(\ow{x})} \sum_{\ow{z} \in \ow{X}} \ow{b}(\ow{x},\ow{z})f_{\ow{z}}(\ow{y}). 
\end{align*}
Since the upper bound is a finite sum of functions which are in $\ell^1(\ow{X}, \ow{m})$, with norm independent of $t$, and since 
\[ \sum_{\ow{y} \in \pi^{-1}(y)} |f(\ow{y}) | = \frac{1}{m(y)} \sum_{\ow{y} \in \pi^{-1}(y)} | f(\ow{y}) | \ow{m}(\ow{y})\]
for any $f \in \ell^1(\ow{X}, \ow{m})$ the first statement follows.

For the second statement, note that by applying Lemma~\ref{l:uniform} again we get that
\[ \lim_{t \to 0^+} \sum_{\ow{y} \in \pi^{-1}(y)} \ow{p}_t(\ow{x},\ow{y}) = \sum_{\ow{y} \in \pi^{-1}(y)} \ow{p}_0(\ow{x},\ow{y})
= \sum_{\ow{y} \in \pi^{-1}(y)} \oh{1}_{\ow{x}}(\ow{y}) = \oh{1}_x(y) = p_0(x,y). \]
\end{proof}

As $p_t(x,y)$ is the minimal non-negative solution to the heat equation on $G$ by \cite{KL12}, it follows that 
\[ p_t(x,y) \leq q_t(x,y). \]  
We now show that the opposite inequality is also true.  In order to do so, we note the following
summation formula.  If $\ow{D} \subseteq \ow{X}$, $D = \pi(\ow{D})$ and $\varphi \in C_c(X)$, then
\begin{equation}\label{e:summation}
 \sum_{y \in D} \left( \sum_{\ow{y} \in \pi^{-1}(y) \cap \ow{D}} \ow{p}_t(\ow{x},\ow{y}) \right)\varphi(y) m(y) = \sum_{\ow{y} \in \ow{D}}\ow{p}_t(\ow{x},\ow{y}) (\varphi \circ \pi)(\ow{y}) \ow{m}(\ow{y}),
\end{equation} 
which is a discrete version of co-area formula for the map $\pi.$

Now, we take any exhaustion sequence $\ow{D}_i$ of $\ow{X}$.  That is, $\ow{D}_i$ are finite, connected, increasing subsets of $\ow{X}$ such that
$\ow{X} = \cup_i \ow{D}_i$.  We let $\ow{p}^i$ denote the Dirichlet heat kernels on $\ow{D}_i$. These satisfy the heat equation on 
\[ \inte \ow{D}_i = \{ \ow{x} \in \ow{D}_i \ | \ \ow{y} \in \ow{D}_i \textup{ for all } \ow{y} \sim \ow{x} \}, \] 
the \emph{interior} of $\ow{D}_i$,
and vanish on the boundary $\partial \ow{D}_i = \ow{D}_i \setminus \inte \ow{D}_i.$ 
For convenience, one may extend $\ow{p}^i$ to the entire graph by setting $\ow{p}^i(\ow{x},\wt{y})=0$ whenever $\ow{x}$ or $\ow{y}\in \ow{X}\setminus \ow{D}_i$.  By maximum principle arguments it follows that $\ow{p}^i \to \ow{p}$ monotonically as $i \to \infty$,
see \cite{Woj08,KL12} for more details. 

We fix $x \in X,$ choose $\ow{x}\in \pi^{-1}(x)$ and  define
\[ q_t^i(y) = \sum_{\ow{y} \in \pi^{-1}(y)\cap \ow{D}_i} \ow{p}_t^i(\ow{x},\ow{y}) \]
and note that by the monotone convergence theorem $q_t^i(y) \to q_t(x,y)$ as $i \to \infty$.
We furthermore note that if $D_i = \pi(\ow{D}_i)$, then $q_t^i(y) = 0$ for $y \in \partial D_i$
and that $q_0^i(y) \leq p_0(x,y)$.  We now show that $q^i$ is a subsolution for the heat equation.
The following is adapted from \cite{Bor00}.

\begin{lemma}\label{l:heat_equation_2}
For $y \in D_i = \pi (\ow{D}_i)$ we have that
\[ (L + \pt) q_t^i(y) \leq 0. \]
\end{lemma}
\begin{proof}
Let $\oh{1}_y = 1_y/m(y)$.  By Green's formula and using the summation equality (\ref{e:summation}) twice, we have that
\begin{align*}
Lq_t^i(y) &= \as{L q_t^i, \oh{1}_y} = \as{q_t^i, L \oh{1}_y} = \sum_{z \in D_i} q_t^i(z) L \oh{1}_y(z) m(z)\\
&= \sum_{z \in D_i} \left( \sum_{\ow{z} \in \pi^{-1}(z) \cap \ow{D}_i} \ow{p}_t^i(\ow{x},\ow{z}) \right) L \oh{1}_y(z) m(z) \\
&= \sum_{\ow{z} \in \ow{D}_i} \ow{p}_t^i(\ow{x},\ow{z}) (L \oh{1}_y \circ \pi)(\ow{z})\ow{m}(\ow{z}) \\
&= \sum_{\ow{z} \in \ow{D}_i} \ow{p}_t^i(\ow{x},\ow{z}) \ow{L}(\oh{1}_y \circ \pi)(\ow{z})\ow{m}(\ow{z}) \\
&= \sum_{\ow{z} \in \ow{D}_i} \ow{L_{\ow{z}}} \ow{p}_t^i(\ow{x},\ow{z}) (\oh{1}_y \circ \pi)(\ow{z})\ow{m}(\ow{z}) \\
&= \sum_{\ow{z} \in \inte \ow{D}_i} -\pt \ow{p}_t^i(\ow{x},\ow{z}) (\oh{1}_y \circ \pi)(\ow{z})\ow{m}(\ow{z}) \\
& \qquad -  \sum_{\ow{z} \in \partial \ow{D}_i} \sum_{\ow{w} \in \ow{D}_i} \ow{b}(\ow{z},\ow{w}) \ow{p}_t^i(\ow{x},\ow{w}) (\oh{1}_y \circ \pi)(\ow{z})\\
&\leq -\pt \sum_{\ow{z} \in \ow{D}_i} \ow{p}^i_t(\ow{x},\ow{z}) (\oh{1}_y \circ \pi)(\ow{z})\ow{m}(\ow{z}) \\
&=- \pt \sum_{z \in D_i}  q_t^i(z) \oh{1}_y(z)m(z) \\
&= - \pt q_t^i(y).
\end{align*}
\end{proof}

\begin{theorem}\label{t:equality}
Let $\ow{G}=(\ow{X},\ow{b},\ow{m})$ with heat kernel $\ow{p}$ be a regular covering of $G=(X,b,m)$ with heat kernel $p$.  
For all $t\geq0$, $x,y \in X$ and $\ow{x} \in \pi^{-1}(x)$, let $q_t(x,y) = \sum_{\ow{y} \in \pi^{-1}(y)} \ow{p}_t(\ow{x},\ow{y})$.
Then,
\[ q_t(x,y) = p_t(x,y). \]
\end{theorem}
\begin{proof}
As $q_t(x,y)$ satisfies the heat equation on $X$ with initial condition $p_0(x,y)$ by Lemma~\ref{l:heat_equation}, 
it follows that $p_t(x,y) \leq q_t(x,y)$ as $p_t(x,y)$ is the minimal non-negative solution.  
On the other hand, by Lemma~\ref{l:heat_equation_2} and using
a maximum principle as, for example Proposition~2.2 in \cite{Woj17}, it follows that $q_t^i(y) \leq p_t(x,y)$.
Since $q_t^i(y) \to q_t(x,y)$ as $i \to \infty$ it follows that $q_t(x,y) \leq p_t(x,y)$ and the conclusion follows.
\end{proof}

We note from the above that it is always true that $\ow{p}_t(\ow{x},\ow{y}) \leq p_t(x,y)$ for all $t \geq 0$, $\ow{x} \in \pi^{-1}(x)$ and 
$\ow{y} \in \pi^{-1}(y)$.  In the case of finitely many sheets, we get that the other inequality holds as well on the diagonal up to a multiple
of the number of sheets.

\begin{lemma}\label{l:heatKernelFiniteSheets}
Let $\ow{G}=(\ow{X},\ow{b},\ow{m})$ be a regular covering of $G=(X,b,m)$ with $n$ sheets. Then,
\[
p_t(x,x) \leq n \cdot \ow{p}_t (\ow{x},\ow{x})
\]
for all $\ow x \in \ow X$ and $x=\pi(\ow x)$.
\end{lemma}
\begin{proof}
We first note that if $\ow{x}_1,\ow{x}_2 \in \pi^{-1}(x)$ for $x \in X$ and if $\gm \in \Gm$ is a deck transformation such that $\gm(\ow{x}_1)=\ow{x}_2$, then
\[ \ow{p}_t(\ow{x}_1,\ow{x}_1) = \ow{p}_t(\gm(\ow{x}_1),\gm(\ow{x}_1)) = \ow{p}_t(\ow{x}_2,\ow{x}_2). \]
Furthermore, as the semigroup is a positive operator, letting $\oh{1}_x = 1_x / m(x)$ we immediately get that
\[ 0 \leq \as{e^{-t \ow{L}}(\oh{1}_{\ow{x}_1} - \oh{1}_{\ow{x}_2}), \oh{1}_{\ow{x}_1} - \oh{1}_{\ow{x}_2} }= \ow{p}_t(\ow{x}_1,\ow{x}_1) + \ow{p}_t(\ow{x}_2,\ow{x}_2) - 2\ow{p}_t(\ow{x}_1, \ow{x_2})\]
so that $\ow{p}_t(\ow{x}_1,\ow{x}_2) \leq \ow{p}_t(\ow{x}_1, \ow{x_1})$.
Therefore, if $\pi^{-1}(x) = \{ \ow{x}_i \}_{i=1}^n$, then by Theorem~\ref{t:equality} we get that
\[ 
 p_t(x,x)  =  \sum_{i=1}^n \ow{p}_t(\ow{x}_1,\ow{x}_i)  \leq n \cdot \ow{p}_t(\ow{x}_1,\ow{x}_1).   \]
\end{proof}

\subsection{Main result}
The equality $q_t(x,y) = p_t(x,y)$ gives our main results on covering of graphs and the heat equation as presented below.  We note that for (ii) and (iii)
the other implication does not hold for the case of coverings with infinite sheets as we will show by example below.
\begin{theorem}\label{t:covering}
Let $\ow{G}= (\ow{X},\ow{b},\ow{m})$ be a regular covering of $G=(X,b,m)$.
\begin{itemize}
\item[(i)]  $\ow{G}$ satisfies (SI) if and only if $G$ satisfies (SI).
\item[(ii)] If $G$ satisfies (FP), then $\ow{G}$ satisfies (FP).
\item[(iii)] If $G$ satisfies (UT), then $\ow{G}$ satisfies (UT).
\end{itemize}
Furthermore, \textup{(ii)} and \textup{(iii)} become equivalences when the regular covering has finitely many sheets.
\end{theorem}
\begin{proof}
For (i), note that by Theorem~\ref{t:equality}
\begin{align*}
\sum_{y \in X} p_t(x,y)m(y) &= \sum_{y \in X} q_t(x,y) m(y) \\
&= \sum_{y \in X}  \sum_{\ow{y} \in \pi^{-1}(y)} \ow{p}_t(\ow{x},\ow{y}) m(y) = \sum_{\ow{y} \in \ow{X}}\ow{p}_t(\ow{x},\ow{y}) \ow{m}(\ow{y})
\end{align*}
from which the conclusion follows immediately.

For (ii), note that Theorem~\ref{t:equality} implies that $\ow{p}_t(\ow{x},\ow{y}) \leq p_t(x,y)$ for all $t\geq 0$ and $\ow{x} \in \pi^{-1}(x), \ow{y} \in \pi^{-1}(y)$.  
Fix some $\ow{x}\in \ow{X}$ and $t>0.$ It suffices to show that for any sequence $\ow{y}_n \to \infty$ in $\ow{X},$ there is a subsequence $\ow{y}_{n_k}$ such that 
\[ \ow{p}_t(\ow{x},\ow{y}_{n_k}) \longrightarrow 0, \quad \mathrm{as}\ k \to \infty. \]
Let $y_n = \pi(\ow{y}_n).$ If $y_n \to \infty$ in $X$, then the result is clear as we assume that $G$ satisfies (FP). If not, one can extract a subsequence, still denoted by $y_n,$ such that $\{ y_n \}_{n=1}^\infty \subseteq D$ for some finite set $D$ in $X.$ We may further choose a subsequence $y_{n_k}$ of $y_n$ such that $\ow{y}_{n_k}$ are contained in the same fiber. It follows
that $\ow{p}_t(\ow{x},\ow{y}_{n_k}) \to 0$ as $k \to \infty$ as these terms are the tail of a convergent series $q_t(x, y_{n_k})$.  Therefore, the conclusion of (ii) follows.

In the case of finitely many sheets, if $\ow{G}$ satisfies (FP) and $y_n \in X$ satisfies $y_n \to \infty$, then for any choice $\ow{y}_n \in \pi^{-1}(y_n)$
it follows that $\ow{y}_n \to \infty$ in $\ow{X}$.  As $\ow{G}$ satisfies (FP), it follows that $\ow{p}_t(\ow{x},\ow{y}_n) \to 0$ as $\ow{y}_n \to \infty$.  Therefore,
by Theorem~\ref{t:equality},
\[ p_t(x,y_n) = q_t(x,y_n) = \sum_{\ow{y}_n \in \pi^{-1}(y_n)} \ow{p}_t(\ow{x},\ow{y}_n) \longrightarrow 0 \]
as $y_n \to \infty$ since the number of terms in the sum above is always equal to the number of sheets.

For (iii), it is clear from the fact that $\ow{p}_t(\ow{x},\ow{y}) \leq p_t(x,y)$ that 
\[ g(x,x) = \int_0^\infty p_t(x,x) dt \geq \int_0^\infty \ow{p}_t(\ow{x},\ow{x}) dt = \ow{g}(\ow{x},\ow{x}) \]
from which the conclusion follows immediately.

Now, assume that the covering has $n$ sheets and that  $\ow{G}$ satisfies (UT) so that $\ow{g}(\ow{x},\ow{x}) \leq C$ for all $\ow{x} \in \ow{X}$.
Letting $x \in X$ and $\ow x \in \pi^{-1}(x)$, by Lemma~\ref{l:heatKernelFiniteSheets} we get that
\[ g(x,x) = \int_0^\infty p_t(x,x) dt \leq \int_0^\infty n \cdot \ow{p}_t(\ow{x},\ow{x}) dt \leq n \cdot \ow{g}(\ow{x},\ow{x}) \leq nC  \]
so that $G$ satisfies (UT).
\end{proof}

The statement of the equality of (FP) on $G$ and $\ow{G}$ in the case of finitely many sheets gives an analogy to Proposition~8.1 in \cite{PS12}. 
However, neither (FP) nor (UT) become equality for the case of infinitely many sheets as the following example shows.
\begin{example}[$\ow{G}$ (UT) $\not \Longrightarrow G$ (FP)]\label{e:converse}
Let 
\[ G = \Z_{1,m} \times C_3 \times C_3\] 
where $C_3$ are ordinary 3-cycles with standard edge weight and measure and
 $\Z_{1,m} = (\Z, b, m)$ with $b(x,y) = 1$ if $|x-y|=1$ and 0 otherwise and
$m(x)=m(-x)=m(r)$ if $|x|=r$ where $m$ satisfies
$ \sum_{r=0}^\infty \sum_{k=r+1}^\infty m(k) <\infty. $  It follows by Theorem~4.13 in \cite{Woj17} that $\Z_{1,m}$ is not Feller.
As such, by Theorem~3.3 in \cite{Woj17}, there exists a function $v \not \in C_0(\Z)$ with $v(0)=1$ and such that $\Lp_{\Z_{1,m}} v(z) = -v(z)$ for all $z \not = 0$. 
This function can be easily extended to $G$ by letting $w(z, x_1, x_2) = v(z)$ for all vertices $(z,x_1,x_2)$ in $G$.  
It follows that $w$ satisfies $\Lp w = -w$ away from $(0,x_1,x_2)$, $w(0,x_1,x_2)=1$ and $w \not \in C_0(X)$.  As such, $G$ does not satisfy (FP) and, 
consequently, does not satisfy (UT).

We now let 
\[\ow{G} = \Z_{1,m} \times \Z \times \Z\] 
where $\Z$ is the ordinary integer lattice with standard edge weight and measure.  As $\Z$ is the universal cover of $C_3$, we
let $\pi_i: \Z \to C_3$ denote the covering maps.  Then, $\pi: \ow{X} \to X$ given by $\pi(z,\ow{x}_1,\ow{x}_2) = (z, \pi_1(\ow{x}_1), \pi_2(\ow{x}_2))$ is a regular
covering.  Topologically, $\ow{G}$ is homeomorphic to $\Z^3$, the standard 3-dimensional integer lattice, and all such graphs are uniformly transient,
independently of the measure, see Corollary~2.6 in \cite{KLSW17}.  Thus, $\ow{G}$ satisfies (UT) and (FP) while $G$ satisfies neither.
\end{example}

\subsection{Spectral applications}
We now give some spectral applications of Theorem~\ref{t:equality}.  In particular, we look at the bottom of the spectrum of the Laplacian which is given by
\[ \lm_0(L) = \inf_{\varphi \in C_c(X), \varphi \not = 0} \frac{\as{L\varphi, \varphi}}{\| \varphi \|^2} \]
and give some connections between the bottom of the spectrum of a graph and its cover.  We also investigate the heat kernel decay
in this setting.

\bigskip

By applying an analogue to a theorem of Li \cite{Li86}, proven in the graph setting in \cite{HKLW12, KLVW15}, we  have the following result concerning
$\lm_0(L)$.
\begin{corollary}
Let $\ow{G}=(\ow{X},\ow{b},\ow{m})$ be a regular covering of $G=(X,b,m)$.  Let $\lm_0(L)$ 
and $\lm_0(\ow{L})$ denote the bottom of
the spectrum of the Laplacian on $G$ and $\ow{G}$, respectively.  Then,
\[ \lm_0(\ow{L}) \geq \lm_0(L). \]
Furthermore, we have equality if the covering has finitely many sheets.
\end{corollary}
\begin{proof}
The heat kernel and the bottom of the spectrum are connected as follows:
\[ \lim_{t \to \infty} \frac{\ln p_t(x,y)}{t} = -\lm_0(L) \]
for any $x,y \in X$, see \cite{HKLW12, KLVW15}.  The result is then immediate since $\ow{p}_t(\ow{x},\ow{y}) \leq p_t(x,y)$
which follows from Theorem~\ref{t:equality}.
In case of finitely many sheets, Lemma~\ref{l:heatKernelFiniteSheets} yields that $p_t(x,x) \leq n \cdot \ow{p}_t (\ow{x},\ow{x})$ where $n$ is the number of sheets. This immediately gives $\lm_0(\ow{L}) \leq \lm_0(L)$ as desired.
\end{proof}

We also have the following analogue to a result of Chavel/Karp, see Corollary~3 in \cite{CK91}.
\begin{corollary}
Let $\ow{G}=(\ow{X},\ow{b},\ow{m})$ be a regular covering of $G=(X,b,m)$.  
Let $\lm_0(\ow{L})$ denote the bottom of the spectrum of the Laplacian on $\ow{G}$.
If the number of sheets of the covering is infinite, then
\[ \lim_{t \to \infty} e^{t\lm_0(\ow{L})}\ow{p}_t(\ow{x},\ow{y}) =0. \]
\end{corollary}
\begin{proof}
It is always true that the limit above exists, see \cite{HKLW12, KLVW15}.  Now, if $\lm_0(\ow{L})=0$, then $\lim_{t \to \infty} \ow{p}_t(\ow{x},\ow{y})=0$ by Corollary~8.2 in \cite{HKLW12}
as $\ow{m}(\ow{X})=\infty$ since the number of sheets is infinite.  

If $\lm_0(\ow{L})>0$ 
and the limit above is positive, then there exists a positive, normalized
eigenfunction $\phi$ to $\lm_0(\ow{L})$ in $\ell^2(\ow{X},\ow{m})$ and it turns out in this case that the eigenspace of $\lm_0(\ow{L})$
is one-dimensional, see \cite{Sul87, HKLW12, KLVW15}.  

By an easy calculation using the invariance under the deck
transformation group $\Gm$, it follows that $\gm^* \phi$ would also be an eigenfunction for $\lm_0(\ow{L})$ in $\ell^2(\ow{X},\ow{m})$
for any $\gm \in \Gm$ where $(\gm^* \phi)(\ow{x}) = \phi(\gm(\ow{x}))$.
As the eigenspace of $\lm_0(\ow{L})$ is one-dimensional, it follows that there exists $\al:\Gamma \to \R_+$ such that
$\gamma^* \phi= \al(\gamma) \phi$ for all $\gamma \in \Gamma$.
Therefore,
\begin{align*}
\| \phi \|^2 &= \sum_{x \in X} \sum_{\ow{x} \in \pi^{-1}(x)} \phi(\ow{x})^2 \ow{m}(\ow{x}) \\
&= \sum_{x \in X}m(x) \sum_{\gm \in \Gm} (\gm^* \phi)^2(\ow{x}) \\
&= \sum_{x \in X}m(x) \phi^2(\ow{x}) \sum_{\gm \in \Gm} (\al(\gm))^2. \\
\end{align*}
As this is independent of the choice of $\ow{x} \in \pi^{-1}(x)$, it follows that
\[ \phi(\ow{x}_1) = \phi(\ow{x}_2) \]
for all $\ow{x}_1, \ow{x}_2 \in \pi^{-1}(x)$.  
Since the measure is independent of the sheet and the number of sheets is infinite, it follows that $\phi$ could not be in $\ell^2(\ow{X},\ow{m})$.
Therefore, the limit must be zero.
\end{proof}

\section{Covering manifolds and heat kernels}\label{sec:manifold}
In this section, we give a deterministic proof of the fact that stochastic incompleteness of a complete Riemannian manifold is equivalent to stochastic incompleteness
of the cover, thus answering a questions raised in \cite{PS12}.  The proof is essentially a synthesis of results found in \cite{Bor00, Li12} and 
the basic argument was already given for graphs in the preceding section though the technicalities are different in the manifold setting.

\bigskip

Let $(M,g)$ be a complete Riemannian manifold. Let $\widetilde{M}$ be a regular covering of $M,$ $\pi:\widetilde{M} \to M$ be the covering map and $\Gamma$ be the deck transformation group. We denote by $\wt{g}$ the lifted Riemannian metric on $\wt{M}$ via the map $\pi.$ 
One can then show that $(\wt{M},\wt{g})$ is complete. 

We denote by $\Delta$ ($\wt{\Delta}$, respectively) the (positive) Laplace-Beltrami operator on $M$ ($\wt{M}$, respectively).
We can then define the heat kernel as follows.  Note that this is essentially the same as in the graph case though the initial condition is given in a distributional sense.
\begin{definition}\label{def:heat kernel}
Let $M$ be a complete Riemannian manifold. We say that $H_t(x,y)$ is a \emph{heat kernel} on $M$ if $H$ is positive, symmetric in the $x$ and $y$ variables and satisfies the heat equation 
\begin{equation}\label{eq:heat}\left(\Delta_x + \pt \right) H_t(x,y)=0 \end{equation} 
for $y\in M$ with initial condition
$$\lim_{t\to 0^+}H_t(x,y)=\delta_y(x)$$ 
where the limit is weak convergence in the sense of measure and $\delta_y(\cdot)$ denotes the point mass delta function at $y.$
\end{definition} 

We let $p_t(x,y)$ ($\wt{p}_t(\wt{x},\wt{y})$, respectively) denote the minimal heat kernel on $M$ ($\wt{M}$, respectively).  These can be
constructed via an exhaustion sequence regardless of the completeness of the manifold, see \cite{Dod83}.
Stochastic incompleteness (SI) is then defined analogous to the case of graphs as
\[ \int_M p_t(x,y) dy < 1 \]
for some (all) $x \in M$ and some (all) $t>0$.

For any $x,y\in M,$ we define 
$$q_t(x,y):=\sum_{\wt{y}\in \pi^{-1}(y)}\wt{p}_t(\wt{x},\wt{y})$$
for $\wt{x}\in\pi^{-1}(x).$ 
By the local Harnack inequality, one can show that $q_t(x,y)$ is finite for any $x,y\in M,$ 
see, for example, the proof of Corollary~16.3 in \cite{Li12}. Since the covering is regular, it is easy to show that the definition 
of $q_t(x,y)$ is independent of choice of $\wt{x}$ in $\pi^{-1}(x)$ and  that $q_t(x,y)$ is symmetric in $x$ and $y.$

Bordoni \cite{Bor00} proved the following estimate, see Lemma~\ref{l:heat_equation_2} above for the proof of the essential step in the case
of graphs.
\begin{proposition}[Proposition~2.4 in \cite{Bor00}] \label{p:bord}
Let $\wt{M}$ be a regular covering of $M.$ Then, for all $t>0,x,y\in M,$
\begin{equation*} 
q_t(x,y)\leq p_t(x,y).
\end{equation*}
\end{proposition}

In fact, an argument by Li yields the following identity, see the proof of Corollary~16.3 in \cite{Li12} for manifolds and Theorem~\ref{t:equality} above for the graph case.
\begin{theorem}\label{thm:pli} Let $\wt{M}$ be a regular covering of $M.$ Then, for all $t>0,x,y\in M,$
$$q_t(x,y)= p_t(x,y).$$
\end{theorem} 

We will give an alternative proof of this result below.  In order to do so, we will show that $q_t(x,y)$ is also a heat kernel in the sense
of Definition~\ref{def:heat kernel}. This is sufficient to prove Theorem~\ref{thm:pli} by combining Bordoni's result with the minimality $p_t(x,y)$.
\begin{proposition}
Let $\wt{M}$ be a regular covering of $M.$ Then, $q_t(x,y)$ is a heat kernel on $M.$
\end{proposition}
\begin{proof}
First, one can use the argument in Theorem~12.4 in \cite{Li12} to verify that $q_t(x,y)$ satisfies the heat equation \eqref{eq:heat} for any $t>0$ by writing 
$$q_t(x,y)=\lim_{k\to\infty}\sum_{i=1}^k\wt{p}_t(\wt{x_i},\wt{y})$$ 
where $\{\wt{x_i}\}_{i=1}^\infty=\pi^{-1}(x).$  

To complete the proof, we show that $q_t(x,y)$ satisfies the initial condition. That is, for any $\varphi\in C_c(M),$ 
where $C_c(M)$ denotes the compactly supported continuous functions on $M$, and any $y\in M$
\begin{equation}\label{eq:initial}
\lim_{t\to 0^+} \int_Mq_t(x,y)\varphi(x)dx = \varphi(y).
\end{equation} 

Without loss of generality, we may assume
that $\varphi \geq 0$. 
It is well-known that there exists a fundamental domain $\wt{M}_1$ in $\wt{M}$ such that $\gamma_1 \wt{M}_1\cap \gamma_2 \wt{M}_1=\emptyset$ if $\gamma_1\neq \gamma_2$ 
for $\gamma_1,\gamma_2\in \Gamma,$ and $\mathrm{vol}(\wt{M}\setminus\cup_{\gamma\in \Gamma}\gamma \wt{M}_1)=0.$
For fixed $y\in M,$ there exists a ball $B_\delta(y)$ of small radius $\delta>0$ such that $B_{2\delta}(\wt{y_1})$ is contained in a fundamental domain $\wt{M}_1$ for some $\wt{y_1}\in \pi^{-1}(y)$ and so that $\pi:B_\delta(\ow{y}_1) \to B_{\delta}(y)$ is an isometry.

Let $\eta$ be a smooth cut-off function on $M$ satisfying $\eta\equiv 1$ on $B_{\frac{\delta}{2}}(y)$, $0 \leq \eta \leq 1$ and $\mathrm{spt}\eta\subset B_\delta(y)$ 
where $\mathrm{spt}\eta$ denotes the support of $\eta.$  We may write
$$\varphi=\eta \varphi+(1-\eta)\varphi.$$ 
Note that since $q_t(x,y) \leq p_t(x,y)$,
\begin{align*}
\int_Mq_t(x,y)\big( (1-\eta)\varphi \big)(x) dx &\leq \int_Mp_t(x,y)\big( (1-\eta)\varphi \big)(x)dx \\
&\to \big( (1-\eta)\varphi \big)(y)=0, \qquad t \to 0^+
\end{align*} 
where we have used the initial condition for $p_t(x,y).$ 
Hence, to prove \eqref{eq:initial}, it suffices to show that 
\begin{equation}\label{eq:initial_2}
\lim_{t\to 0^+} \int_Mq_t(x,y)(\eta\varphi)(x)dx = \varphi(y).
\end{equation}

For simplicity, we write $\zeta=\eta \varphi.$ We note that $\mathrm{spt}\zeta\subset B_\delta(y)$ and  $\zeta(y)=\varphi(y).$ 
As we assume $\varphi\geq0$, by the fact that $q_t(x,y) \leq p_t(x,y)$, we get
\begin{align*}
\int_{M}q_t(x,y)\zeta(x)dx-\zeta(y)&\leq \int_M p_t(x,y)\zeta(x)dx-\zeta(y) \\
& \to 0, \qquad t \to 0^+
\end{align*}

On the other hand, to estimate the opposite difference, we use the fact that
$\mathrm{spt}\zeta\subset B_\delta(y)$ and $B_{\delta}(\wt{y_1})\subset \wt{M}_1,$ to lift the function $\zeta$ to $\wt{M}_1,$ by
$$\wt{\zeta}(\wt{x})=\left\{\begin{array}{ll}\zeta(\pi(\wt{x})),&\wt{x}\in \wt{M}_1,\\0,&\wt{x}\in \wt{M}\setminus \wt{M}_1.\end{array} \right.$$
Then $\wt{\zeta}\in C_c(\wt{M})$ with $\mathrm{spt}\wt{\zeta}\subset\wt{M_1}.$ Hence,
\begin{align*}
\zeta(y)-\int_{M}q_t(x,y)\zeta(x)dx &= \wt{\zeta}(\wt{y_1})-\int_{\mathrm{spt} \zeta}\sum_{\wt{x}\in \pi^{-1}(x)}\wt{p}_t(\wt{x},\wt{y_1})\zeta(x)dx\\
&\leq \wt{\zeta}(\wt{y_1})-\int_{\ow{M}_1}\wt{p}_t(\wt{x},\wt{y_1})\wt{\zeta}(\wt{x})d\wt{x}\\
& \to 0, \qquad t \to 0^+.
\end{align*} 

By combining these two inequalities, we get \eqref{eq:initial_2} which completes the proof.
 \end{proof}                                                           

We are now ready to give another proof of Theorem~\ref{thm:pli}, which was originally proven by Li using the Duhamel principle, 
see the proof of Corollary~16.3 in \cite{Li12}. 
\begin{proof}[Proof of Theorem~\ref{thm:pli}] Since $p_t(x,y)$ is the minimal heat kernel and $q_t(x,y)$ is a heat kernel by the proposition above, we get 
$$p_t(x,y)\leq q_t(x,y)$$ 
by Theorem~3.6 in \cite{Dod83}.
This proves the theorem by combining it with Bordoni's result, Proposition~\ref{p:bord}.
\end{proof}

In the case when $\wt{M}$ is a regular covering of $M$, Elworthy \cite{Elw82} used stochastic differential equations to give a proof of the fact that $M$ satisfies 
(SI) if and only if $\wt{M}$ satisfies (SI). 
Pigola and Setti \cite{PS12} asked for a deterministic proof of this fact. 
By Theorem~\ref{thm:pli}, we may provide an affirmative answer to their question.
\begin{theorem}\label{thm:scequ} Let $\wt{M}$ be a regular covering of $M.$ 
$M$ satisfies (SI) if and only if $\wt{M}$ satisfies (SI).
\end{theorem}  
\begin{proof} 
Given $x\in M,$ we choose any $\wt{x}\in \pi^{-1}(x).$ By Theorem ~\ref{thm:pli},
\begin{align*}
\int_M p_t(x,y)dy &= \int_M q_t(x,y)dy \\
&=\int_M\sum_{\wt{y}\in \pi^{-1}(y)}\wt{p}_t(\wt{x},\wt{y})d\wt{y} =\int_{\wt{M}}\wt{p}_{t}(\wt{x},\wt{y})d\wt{y}
\end{align*} 
where the last equality follows from the co-area formula. The theorem is a direct consequence of the above equality.
\end{proof}

\bibliographystyle{alpha}
\bibliography{coverings.bib}{}

\end{document}